\documentclass[3p]{elsarticle}

\usepackage[english,activeacute]{babel}
\usepackage[T1]{fontenc}
\usepackage{amsmath,amssymb}

\newtheorem{remark}{Remark}[section]

\newtheorem{theorem}{Theorem}[section]

\newtheorem{proposition}{Proposition}[section]

\newproof{proof}{Proof}
\newproof{dem-final}{Proof of Lemma \ref{lema-multiplicador}}
\newtheorem{notation}{Notation}
\def\d{\hbox{d}}

\begin{document}
\begin{frontmatter}
\title{Behavior of the plastic deformation of an elasto-perfectly-plastic oscillator with noise.\footnote{This research was partially supported by a grant from CEA, Commissariat \`a l'\'energie atomique and by the National Science Foundation under grant DMS-0705247. A large part of this work was completed while one of the authors was visiting the University of Texas at Dallas and the Hong-Kong Polytechnic University. We wish to thank warmly these institutions for the hospitality and support.}}


\selectlanguage{english}
\author[authorlabel1]{Alain Bensoussan},
\ead{alain.bensoussan@utdallas.edu}
\author[authorlabel2]{Laurent Mertz}
\ead{mertz@ann.jussieu.fr}

\address[authorlabel1]{International Center for Decision and Risk Analysis, School of Management, University of Texas at Dallas, Box 830688, Richardson, Texas 75083-0688,
Graduate School of Business, the Hong Kong Polytechnic University,
Graduate Department of Financial Engineering, Ajou University.
This research in the paper was supported by WCU (World Class University) program through the National Research Foundation of Korea funded by the Ministry of Education, Science and Technology (R31 - 20007).}
\address[authorlabel2]{Universit\'e Pierre et Marie Curie-Paris 6,
Laboratoire Jaques Louis Lions, 4 place jussieu 75005 Paris.}
\begin{abstract}
{Earlier works in engineering, partly experimental, partly computational have revealed that asymptotically, when the excitation is a white noise, plastic deformation and total deformation for an elasto-perfectly-plastic oscillator have a variance which increases linearly with time with the same coefficient. 
In this work, we prove this result and we characterize the corresponding drift coefficient.
Our study relies on a stochastic variational inequality governing the evolution between the velocity of the oscillator and the non-linear restoring force. 
We then define long cycles behavior of the Markov process solution of the stochastic variational inequality which is the key concept to obtain the result.
An important question in engineering is to compute this coefficient.
Also, we provide numerical simulations which show successful agreement with our theoretical prediction and previous empirical studies made by engineers.

\vskip 0.5\baselineskip
\noindent{\bf R\'esum\'e}
\vskip 0.5\baselineskip

{\bf Le comportement de la d\'eformation plastique pour un oscillateur \'elastique-parfait\-ement-plastique excit\'e par un bruit blanc.}
Des r\'esultats exp\'erimentaux en sciences de l'ing\'enieur ont montr\'e que, pour un oscillateur \'elasto-plastique-parfait excit\'e par un bruit blanc, la d\'eformation plastique et la d\'eformation totale ont une variance, qui asymptotiquement, cro\^it lin\'eairement avec le temps avec le m\^eme coefficient.
Dans ce travail, nous prouvons ce r\'esultat et nous caract\'erisons le coefficient de d\'erive.
Notre \'etude repose sur une in\'equation variationnelle stochastique gouvernant l'\'evolution entre la vitesse de l'oscillateur et la force de rappel non-lin\'eaire. 
Nous d\'efinissons alors le comportement en cycles longs du processus de Markov solution de l'in\'equation variationnelle stochastique qui est le concept cl\'e pour obtenir le r\'esultat. Une question importante en sciences de l'ing\'enieur est de calculer ce coefficient.
Les r\'esultats num\'eriques confirment avec succ\`es notre pr\'ediction th\'eorique et les \'etudes empiriques faites par les ing\'enieurs.}
\end{abstract}

\end{frontmatter}

\section*{Version fran\c{c}aise abr\'eg\'ee}
Dans cet article, nous \'etudions la variance de l'oscillateur \'elastique-parfaitement-plastique (EPP) excit\'e par un bruit blanc. La dynamique de l'oscillateur s'exprime \`a l'aide d'une \'equation \`a m\'emoire. (voir \eqref{chap3:hysteresis1}-\eqref{chap3:hysteresis2}). A.Bensoussan and J.Turi  ont montr\'e que la relation entre la vitesse et la composante \'elastique satisfait une in\'equation variationnelle stochastique (voir \eqref{chap3:svi}). Dans ce cadre, nous introduisons les cycles long ind\'ependants (d\'efinis plus loin) et nous justifions qu'ils permettent de caract\'eriser la variance de la d\'eformation totale et de la d\'eformation plastique (voir \eqref{chap3:sigma}). 
\selectlanguage{english}
\section{Introduction} 
In civil engineering, an elasto-perfectly-plastic (EPP) oscillator with one single degree of freedom is employed to estimate prediction of failure of mechanical structures under random vibrations.
This elasto-plastic oscillator consists in a one dimensional model simple and representative of the elasto-plastic behavior of a class of structure dominated by their first mode of vibration. The main difficulty to study these systems comes from a frequent occurrence of plastic phases on small intervals of time. A plastic deformation is produced when the stress of the structure crosses over an elastic limit. The dynamics of the EPP-oscillator has memory, so it has been formulated in the engineering literature as a process with hysteresis $x(t)$, which stands for the displacement of the oscillator. We study the problem
\begin{equation}
\label{chap3:hysteresis1}
\ddot{x} + c_0 \dot{x} + \mathbf{F}_t = \dot{w}
\end{equation}
with initial conditions of displacement and velocity $x(0) = 0, \quad \dot{x}(0) = 0$.
Here $c_0>0$ is the viscous damping coefficient, $k>0$ the stiffness, $w$ is a Wiener process; $\mathbf{F}_t := \mathbf{F}(x(s), 0 \leq s \leq t)$ is a non-linear functional which depends on the entire trajectory $\{ x(s), 0 \leq s \leq t\}$ up to time $t$. The plastic deformation denoted by $\Delta(t)$ at time $t$ can be recovered from the pair $(x(t),\mathbf{F}_t)$ by the following relationship:
\begin{equation}\label{chap3:hysteresis2}
\mathbf{F}_t = 
\left\{
\begin{array}{rcl}
k Y & if & x(t) = Y + \Delta(t),\\
k (x(t) - \Delta(t)) & if & x(t)  \in ]-Y+ \Delta(t),Y+ \Delta(t)[,\\
-k Y & if & x(t) = -Y + \Delta(t).\\
\end{array}
\right.
\end{equation}
where $\Delta(t) = \int_0^t y(s)\mathbf{1}_{\{ \vert \mathbf{F}_t \vert = kY\}} \d s $ and $Y$ is an elasto-plastic bound.
Karnopp \& Scharton \cite{KarSchar} proposed a separation between elastic states and plastic states. They introduced a fictitious variable $z(t) := x(t) - \Delta(t)$ and noticed the simple fact that between two plastic phases $z(t)$ behaves like a linear oscillator.
In addition, other works made by engineers \cite{BorLab}, partly experimental, partly computational, have revealed that total deformation has a variance which increases linearly with time:  
\begin{equation}\label{chap3:sigma2}
\lim_{t \to \infty } \frac{\sigma^2(x(t))}{t} = \sigma^2. 
\end{equation}
where $\sigma^2$ is a positive real number.\\

Recently, the right mathematical framework of stochastic variational inequalities (SVI) modeling an elasto-plastic oscillator with noise (presented below) has been introduced by the first author with J. Turi in \cite{BenTuri1}. The inequality governs the relationship between the velocity $y(t)$ and the elastic component $z(t)$:  
\begin{equation}\label{chap3:svi}\tag{$\mathcal{SVI}$}
\d y(t)  = -(c_0 y(t) + k z(t)) \d t + \d w(t), \quad (\d z(t)-y(t) \d t)(\phi - z(t)) \geq 0, \quad  \forall \vert \phi \vert \leq Y, \quad \vert z(t) \vert \leq Y.
\end{equation}
The plastic deformation $\Delta(t)$ can be recovered by $\int_0^t y(s)\mathbf{1}_{\{ \vert z(s) \vert =Y\}}\d s$.\\

Throughout the paper, the objective is to prove \eqref{chap3:sigma2} and to provide a exact characterization of \eqref{chap3:sigma2} based on \eqref{chap3:svi}. Now, we introduce \textit{long cycles} in view to the fact that we are interested in identifying independent sequences in the trajectory.

\subsection{Long cycles}
Denote $\tau_0 := \inf \{ t > 0, \quad y(t) = 0 \quad \mbox{and} \quad \vert z(t) \vert = Y \}$ and $s:=\textup{sign}(z(\tau_0))$ which labels the first boundary hit by the process $(y(t),z(t))$. Define
$\theta_0 := \inf \{ t > \tau_0, \quad y(t) = 0 \quad \mbox{and} \quad z(t) = -s Y \}$.
In a recurrent manner, knowing $\theta_n$, we can define for $n \geq 0$
\begin{align*}
\tau_{n+1} & := \inf \{ t > \theta_n, \quad y(t) = 0 \quad \mbox{and} \quad z(t) = s Y \},\\
\theta_{n+1} & := \inf \{ t > \tau_{n+1}, \quad y(t) = 0 \quad \mbox{and} \quad z(t) = -s Y \}.
\end{align*}
Now, according to the previous setting we can define the $n$-\textit{th} long cycle (resp. first part of the cycle, second part of the cycle) as the piece of trajectory delimited by the interval $[\tau_{n},\tau_{n+1})$, (resp. $[\tau_{n},\theta_{n+1})$ and $[\theta_{n+1},\tau_{n+1})$). 
Indeed, at the instants $\{ \tau_{n}, n \geq 1 \}$, the process $(y(t),z(t))$ is in the same state at the instant $\tau_0$. In addition, there are two types of cycles depending on the sign of $s$. The set of stopping times $\{ \tau_{n}, n \geq 0 \}$ represents the times of occurence of long cycles. Let us remark that the plastic deformation and the total deformation are the same on a long cycle since $\int_{\tau_0}^{\tau_1} y(t)\d t = \int_{\tau_0}^{\tau_1} y(t)\mathbf{1}_{\{ \vert z(s) \vert =Y\}} \d t  + \int_{\tau_0}^{\tau_1} y(t)\mathbf{1}_{\{ \vert z(s) \vert <Y\}} \d t$ and that 
$\int_{\tau_0}^{\tau_1} y(t)\mathbf{1}_{\{ \vert z(s) \vert <Y\}} \d t = z(\tau_1) - z(\tau_0) = 0$.
As main result, we have obtained the following theorem.
\begin{theorem}[Characterization of the variance related to the plastic/total deformation]
In the previously defined context, we have shown
\begin{equation}\label{chap3:sigma}
\lim_{t \to \infty } \frac{\sigma^2(x(t))}{t} = \frac{\mathbb{E} \left ( \int_{\tau_0}^{\tau_1} y(t) \d t \right)^2}{\mathbb{E} (\tau_1 - \tau_0)}.
\end{equation}
\end{theorem}
Our proofs are based on solving nonlocal partial differential equations related to long cycles. Simpler formula will be given below at equation \eqref{chap3:last}.
\section{The issue of long cycles and plastic deformations}
Let us introduce notations.
\begin{notation}
$D := \mathbb{R} \times (-Y,+Y), \quad D^+ := (0,\infty) \times \{Y \}, \quad D^- := (-\infty,0) \times \{-Y \}$, 
and the differential operators
$A \zeta :=    -\frac{1}{2} \zeta_{yy} + (c_0y+kz)  \zeta_{y}  - y \zeta_z, \quad B_+\zeta :=  -\frac{1}{2} \zeta_{yy} + (c_0y + kY)  \zeta_y, \quad 
B_-\zeta :=   -\frac{1}{2} \zeta_{yy} + (c_0y - kY)  \zeta_y.$
where $\zeta$ is a regular function on $D$.
\end{notation}
Let $f$ be a bounded measurable function, we want to solve
\begin{equation}\tag{$P_{v^+}$}
A v^+    =  f(y,z)  \quad  \textup{in} \quad D, \quad 
B_+ v^+   =  f(y,Y)  \quad  \textup{in} \quad D^+, \quad 
B _- v^+   =  f(y,-Y) \quad \textup{in} \quad D^-
 \label{chap3:vplus}
\end{equation}
with the nonlocal boundary conditions $v^+(y,Y)$ continuous and $v^+(0^-,-Y) =0$,
and 
\begin{equation}\tag{$P_{v^-}$}
A v^-     =  f(y,z)  \quad  \textup{in} \quad D, \quad 
B_+ v^-   =  f(y,Y)  \quad  \textup{in} \quad D^+, \quad 
B _- v^-    =  f(y,-Y) \quad \textup{in} \quad D^-
 \label{chap3:vmoins}
\end{equation}
with the nonlocal boundary conditions $v^-(0^+,Y)=0$ and $v^-(y,-Y)$ continuous.
The functions $v^+(y,z;f)$ and $v^-(y,z;f)$ are called \underline{long cycles}.
In addition, we define $\pi^+(y,z)$ and $\pi^-(y,z)$ by
\begin{equation}\tag{$P_{\pi^+}$}
A \pi^+    =  0  \quad  \textup{in} \quad  D,\quad
\pi^+(y,Y)   =  1  \quad \textup{in} \quad D^+,\quad 
\pi^+(y,-Y)   =  0 \quad \textup{in} \quad D^-
\label{chap3:piplus}
\end{equation}
and
\begin{equation}\tag{$P_{\pi^-}$}
A \pi^-   =  0  \quad \textup{in} \quad D,\quad
\pi^-(y,Y)   =  0  \quad \textup{in} \quad  D^+,\quad
\pi^-(y,-Y)  =  1  \quad \textup{in} \quad D^-
\label{chap3:pimoins}
\end{equation}
Note that  $\pi^+(y,z) + \pi^-(y,z) =1$, so the existence and uniqueness of a bounded solution \eqref{chap3:piplus} and \eqref{chap3:pimoins} is clear. 
\begin{proposition}
We have the properties $\pi^-(0^+,-Y)>0$ and $\pi^+(0^-,Y)>0$.
\end{proposition}
\begin{proof}
We check only the first property. We consider the elastic process $(y_{yz}(t),z_{yz}(t))$:
\begin{align*}
z_{yz}(t)  & = e^{\frac{-c_0t}{2}} \lbrace z \cos{(\omega t )} + \frac{1}{\omega}(y+\frac{c_0}{2}z) \sin{(\omega t)}\rbrace +  \frac{1}{\omega}  \int_0^t e^{-\frac{c_0}{2}(t-s)} \sin{(\omega(t-s))} \d w(s),\\
y_{yz}(t) & =-\frac{c_0}{2}z_{yz}(t)+ e^{-\frac{c_0 t}{2}} \lbrace -\omega z \sin{(\omega t)} + (y+\frac{c_0}{2}z)\cos{(\omega t)}\rbrace +\int_0^t e^{-\frac{c_0}{2}(t-s)}\cos{(\omega (t-s))} \d w(s)
\end{align*}
where $\omega  := \frac{\sqrt{4k - c_0^2}}{2}$ (we assume $4k >c_0^2$).
Note that the condition $4k >c_0^2$ is needed so that $(y_{yz}(t),z_{yz}(t))$ have real valued solutions.
Set  $\tau_{yz} := \inf \{ t>0, \quad \vert z_{yz}(t) \vert \geq Y \}$
then we have the probabilistic interpretation
$\pi^+(y,z) = \mathbb{P}(z_{yz}(\tau_{yz}) = Y), \quad  \pi^-(y,z) = \mathbb{P}(z_{yz}(\tau_{yz}) = -Y)$.
We can state
\begin{equation}\label{chap3:e10}
\pi^-(y,z) \to 1 \quad \mbox{as} \quad y \to -\infty, \quad z \in [-Y,Y].
\end{equation}
Indeed, $\forall t$ with $0<t<\frac{\pi}{\omega}$ we have $z_{yz}(t) \to - \infty$, as $y \to -\infty$ a.s.
Therefore $\forall t$ with $0< t < \frac{\pi}{\omega}$, a.s. $z_{yz}(t) < -Y$ for $y$ sufficiently large. 
Hence, a.s. $\tau_{yz} < t$ for $y$ sufficiently large.
Therefore a.s. $\limsup_{y \to -\infty} \tau_{yz} < t$.
Since $t$ is arbitrary, necessarily a.s. $\tau_{yz} \to 0$, as $y \to -\infty$
which implies \eqref{chap3:e10}.
Moreover the function $\pi^+(y,Y)$ cannot have a minimum or a maximum at any finite $y <0$. It is then monotone and since $\pi^-(-\infty,Y) = 1$, it is monotone decreasing. It follows that 
$\pi^-(0^-,Y) <1$. It cannot be $0$. Otherwise, $\pi^-(y,Y)$ is continuous at $y=0$, and $(0,Y)$ is a point of minimum of $\pi^-(y,z)$. Since for $y<0$, $\pi_y^-(y,Y) <0$ from the equation of $\pi^-$ we get
$\limsup_{y \to 0} \pi_{yy}^-(y,Y) \leq 0$
which is not possible since $(0,Y)$ is a minimum.
\end{proof}
We next define 
$\eta(y,z)$ by
\begin{equation}\tag{$P_{\eta}$}
-\frac{1}{2} \eta_{yy} + (c_0y + kz)\eta_y - y\eta_z  = f(y,z)  \quad \textup{in} \quad D,\quad
\eta(y,Y)  = 0  \quad \textup{in} \quad D^+,\quad
\eta(y,-Y)  =  0 \quad \textup{in} \quad D^-
\label{chap3:eta}
\end{equation}
with the local boundary conditions $\eta(0^+,Y)=0$ and $\eta(0^-,-Y)=0$.
For $f$ bounded \eqref{chap3:eta} has a unique bounded solution.
Then define $\varphi_+(y;f)$ by solving
\begin{equation}
\label{chap3:e11}
-\frac{1}{2} \varphi_{+,yy} + (c_0 y + k Y) \varphi_{+,y} = f(y,Y), \quad y >0, \quad \varphi_+(0;f) = 0.
\end{equation}
We check easily the formula
\begin{equation}\label{chap3:e12}
\varphi_+(y;f) = 2 \int_0^{\infty} \d \xi \exp(-(c_0 \xi^2 + 2 kY \xi)) \int_{\xi}^{\xi + y} f(\zeta;Y) \exp(-2 c_0 \xi (\zeta - \xi)) \d \zeta, \quad y \geq 0.
\end{equation}
We consider $\psi_+(y,z;f)$ defined by
\begin{equation}\tag{$P_{\psi_+}$}
\left\{
\begin{array}{rcl}
-\frac{1}{2} \psi_{+,yy} + (c_0y + kz)\psi_{+,y} - y\psi_{+,z}  &= &0  \quad \textup{in} \quad D,\\
\psi_+(y,Y)  &= &\varphi_+(y;f)  \quad \textup{in} \quad D^+,\\ 
\psi_+(y,-Y)  &= &0 \quad \textup{in} \quad D^-
\end{array}
\right.
\label{chap3:psiplus}
\end{equation}
and similarly $\varphi_-(y;f), \psi_-(y,z;f)$ defined by
\begin{equation}
\label{chap3:e12bis}
-\frac{1}{2} \varphi_{-,yy} + (c_0 y - k Y) \varphi_{-,y} = f(y,-Y), \quad y <0, \quad \varphi_-(0;f) = 0.
\end{equation}
which leads to
\[
\varphi_-(y;f) = 2 \int_0^{\infty} \d \xi \exp(-(c_0 \xi^2 - 2 kY \xi)) \int_{y-\xi}^{-\xi} f(\zeta;-Y) \exp(-2 c_0 \xi (\zeta - \xi)) \d \zeta, \quad y \leq 0.
\]
and
\begin{equation}\tag{$P_{\psi_-}$}
\left\{
\begin{array}{rcl}
-\frac{1}{2} \psi_{-,yy} + (c_0y + kz)\psi_{-,y} - y\psi_{-,z}  &= &0  \quad \textup{in} \quad D,\\
\psi_-(y,Y) & = & 0  \quad \textup{in} \quad  D^+,\\ 
\psi_-(y,-Y) & = & \varphi_-(y;f) \quad  \textup{in} \quad  D^-\\
\end{array}
\right.
\label{chap3:psimoins}
\end{equation}
We can state 
\begin{proposition}
The solution of \eqref{chap3:vplus} is given by 
\begin{align}\label{chap3:e17}
v^+(y,z;f)  = & \eta(y,z;f) + \psi_+(y,z;f) + \psi_-(y,z;f)\\ 
&  + \frac{\eta(0^-,Y;f) + \psi_+(0^-,Y;f) + \psi_-(0^-,Y;f)}{\pi^-(0^-,Y)} \pi^+(y,z). \nonumber
\end{align}
and the solution of \eqref{chap3:vmoins} is given by
\begin{align}
v^-(y,z;f)  = & \eta(y,z;f) + \psi_+(y,z;f) + \psi_-(y,z;f)\\ 
&  + \frac{\eta(0^+,-Y;f) + \psi_+(0^+,-Y;f) + \psi_-(0^+,-Y;f)}{\pi^+(0^-,Y)} \pi^-(y,z). \nonumber
\end{align}
\end{proposition}
\begin{proof}
Direct checking. 
\end{proof}

\section{Complete cycle}
First, let us check that $\mathbb{E}[x(t)] =0$ and then $\sigma^2(x(t)) = \mathbb{E}[x^2(t)]$. Indeed, by symmetry of the inequality \eqref{chap3:svi} and by the choice of initial conditions $y(0)=0, z(0)=0$, the processes $(y(t),z(t))$ and $-(y(t),z(t))$ have same law. Then, $\mathbb{E}[x(t)] =\mathbb{E} \left [\int_0^t y(s) \d s \right ] =0$.
In addition, $(y(\tau_1),z(\tau_1))$ is equal to $(0,-Y)$ or $(0,Y)$ with probability $1/2$ for both, therefore with no loss of generality we can suppose that
$(y(0),z(0)) = (0,-Y)$ or $(0,Y)$ with probability $1/2$ for both. Let us treat the case $y(0)=0, z(0)=Y$. So, $\tau_0=0$ and $\theta_1 = \inf \{ t >0, \quad z(t)=-Y, y(t) = 0 \}$.
We can assert that $\mathbb{E}\theta_1 = v^+(0,Y;1)$
hence $\theta_1 < \infty$ a.s..
Next we define $\tau_1 = \inf \{ t > \theta_1, \quad z(t)=Y, y(t) = 0 \}$
then
$\mathbb{E}\tau_1 = v^+(0,Y;1) + v^-(0,-Y;1) = 2 v^+(0,Y;1)$.
At time $\tau_1$ the state of the system is again $(0,Y)$. So the sequence $\{ \tau_n , n \geq 0 \}$ is such that $\tau_n<\tau_{n+1}$ and in the interval $(\tau_n,\tau_{n+1})$ we have a cycle identical to $(0,\tau_1)$. Consider the variable $\int_0^{\tau_1} y(t) \d t$. We have $\mathbb{E} \int_0^{\tau_1} y(t) \d t = \mathbb{E} \int_0^{\theta_1} y(t) \d t + \mathbb{E} \int_{\theta_1}^{\tau_1} y(t)dt = v^+(0,Y;y) + v^-(0,-Y;y)$. However if $f$ is antisymmetric $f(-y,-z) = -f(y,z)$ we have $v^+(0,Y;f) = - v^-(0,-Y;f)$,
therefore $\mathbb{E} \int_0^{\tau_1} y(t) \d t = 0$.
Hence, $\mathbb{E} \left ( \int_0^{\tau_n} y(t) \d t \right )^2 = \mathbb{E} \left ( \sum_{j=0}^{n-1} \int_{\tau_j}^{\tau_{j+1}} y(t) \d t \right )^2 = n \mathbb{E} \left ( \int_0^{\tau_1} y(t) \d t \right )^2$.
Then
\[
\frac{\mathbb{E} \left ( \int_0^{\tau_n} y(t) dt \right )^2}{\mathbb{E} \tau_n} = \frac{\mathbb{E} \left ( \int_0^{\tau_1} y(t) dt \right )^2}{\mathbb{E} \tau_1}. 
\]
Let $T>0$ and $N_T$ with $\tau_{N_T} \leq T < \tau_{N_{T+1}}, \quad N_T = 0 \quad \mbox{if} \quad \tau_1 > T, \quad \tau_0 = 0$ and 
then calculations lead to the following
\[
\frac{\mathbb{E} \left ( \int_0^{\tau_{N_T+1}} y(t) \d t \right )^2}{\mathbb{E} \tau_{N_T+1} } = \frac{\mathbb{E} \left ( \int_0^{\tau_1} y(t) \d t \right )^2}{\mathbb{E} \tau_1}.
\]
Next, we can justify that
\begin{equation}\label{chap3:lim}
\lim_{T \to \infty} \frac{1}{T} \mathbb{E} \left ( \int_0^T y(t) \d t \right )^2 = \lim_{T \to \infty} \frac{1}{T} \mathbb{E} \left ( \int_0^{\tau_{N_T+1}} y(t) \d t  \right )^2
\end{equation}
and that we have  a lower bound and an upper bound for \eqref{chap3:lim}, that is  
\[
\frac{ \mathbb{E} \left ( \int_0^{\tau_1} y(t) \d t \right )^2}{\mathbb{E} \tau_1} \leq \frac{1}{T} \mathbb{E} \left ( \int_0^{\tau_{N_T+1}} y(t) \d t \right )^2 \leq \left ( 1 + \frac{\mathbb{E}\tau_1}{T} \right ) \frac{ \mathbb{E} \left ( \int_0^{\tau_1} y(t) \d t \right )^2}{\mathbb{E} \tau_1}. 
\]
Therefore
\begin{equation}\label{chap3:e38}
\lim_{T \to \infty} \frac{1}{T} \mathbb{E} \left ( \int_0^T y(t) \d t \right )^2 = \frac{\mathbb{E} \left ( \int_0^{\tau_1} y(t) \d t \right )^2}{\mathbb{E} \tau_1}.
\end{equation}
Moreover, we can simplify \eqref{chap3:e38}, indeed we have
\[
\mathbb{E} \left ( \int_0^{\tau_1} y(t) \d t \right )^2 = \mathbb{E} \left [ \left ( \int_0^{\theta_1} y(t) \d t \right )^2 + \left ( \int_{\theta_1}^{\tau_1} y(t) \d t \right )^2 + 2\int_0^{\theta_1} y(t) \d t \int_{\theta_1}^{\tau_1} y(t) \d t \right ]
\]
and therefore, we can justify that
\[
\mathbb{E} \left ( \int_0^{\tau_1} y(t) \d t \right )^2 = 2 \left [ \mathbb{E} \left ( \int_0^{\theta_1} y(t) \d t \right )^2 - (\int_0^{\theta_1} y(t) \d t)^2 \right ] = 2 \left [ \mathbb{E} \left ( \int_0^{\theta_1} y(t) \d t \right )^2 - (v^+(0,Y;y))^2 \right ] .
\]
Since  $\mathbb{E} \tau_1 = 2 v^+(0,Y;1)$
we obtain
\begin{equation}
\label{chap3:last}
\lim_{T \to \infty} \frac{1}{T} \mathbb{E} \left ( \int_0^T y(t) \d t \right )^2 = \frac{\mathbb{E} \left ( \int_0^{\theta_1} y(t) \d t \right )^2-(v^+(0,Y;y))^2}{v^+(0,Y;1)}.
\end{equation}
\section{Numerical evidence in support of our result}
In this section, we provide computational results which confirm our theoretical results. 
A C code has been written to simulate $(y(t),z(t))$. See \cite{BenMerPiTu} for the numerical scheme considered to do direct simulation.
Let $T>0$, $N \in \mathbb{N}$ and $\{ t_n=n \delta t, 0 \leq n \leq N \}$ where $\delta t= \frac{T}{N}$. 
Then, to compute the left hand side of \eqref{chap3:sigma}, we consider $MC \in \mathbb{N}$ and we generate $MC$ numerical solutions of \eqref{chap3:svi} $\{ y^i(t),0 \leq t \leq T,  1 \leq i \leq MC \}$ up to the time $T$. By the law of large numbers, we can approximate $\frac{1}{T} \mathbb{E} \left ( \int_0^T y(s) ds \right )^2$ by $X_{MC}:= \frac{1}{T}\frac{1}{MC} \sum_{i=1}^{MC} \left ( \sum_{i=1}^{N} y^i(t_n) \delta t  \right )^2$ and also
$\frac{1}{T^2} \mathbb{E} \left ( \int_0^T y(s) ds \right )^4$ by $X_{MC}^2:= \frac{1}{T^2}\frac{1}{MC} \sum_{i=1}^{MC} \left ( \sum_{i=n}^{N} y^i(t_n) \delta t  \right )^4$.
Denote $s_X := \sqrt{X_{MC}^2- (X_{MC})^2}$, we also know by the central limit theorem that $\frac{1}{T} \mathbb{E} \left ( \int_0^T y(s) ds \right )^2 \in [X_{MC} - \frac{1.96 s_X}{\sqrt{MC}}, X_{MC} + \frac{1.96 s_X}{\sqrt{MC}}]$ with $95 \%$ of confidence. Similarly, to compute the right hand side of \eqref{chap3:sigma}, we generate $MC$ numerical long cycles. For each trajectory $\{ y^i(t),t \geq 0 \}$, we consider $N_c^i$ the required number of time step to obtain a completed cycle. Denote $\delta_{MC} := \frac{1}{MC} \sum_{i=1}^{MC} \left ( \sum_{n=0}^{N_c^i} y^i(t_n) \delta t \right )^2$, $\delta_{MC}^2 := \frac{1}{MC} \sum_{i=1}^{MC} \left ( \sum_{n=0}^{N_c^i} y^i(t_n) \delta t \right )^4$, $\tau_{MC} := \frac{1}{MC} \sum_{i=1}^{MC}{ N_c^i \delta t}$ , $\tau_{MC}^2 := \frac{1}{MC} \sum_{i=1}^{MC}{ (N_c^i \delta t)^2}$, $s_\delta := \sqrt{\delta_{MC}^2- (\delta_{MC})^2}$ and $s_\tau := \sqrt{\tau_{MC}^2- (\tau_{MC})^2}$. We also know that $\frac{\delta_{MC}}{\tau_{MC}} \in [\frac{\delta_{MC} - \frac{1.96 s_\delta}{\sqrt{MC}}}{\tau_{MC} + \frac{1.96 s_\tau}{\sqrt{MC}}}, \frac{\delta_{MC} + \frac{1.96 s_\delta}{\sqrt{MC}}}{\tau_{MC} - \frac{1.96 s_\tau}{\sqrt{MC}}}]$ with $95 \%$ of confidence. In table \ref{chap3:tablesigma}, is shown a comparison of the results obtained for $X_{MC}$ and $\frac{\delta_{MC}}{\tau_{MC}}$ where $T = 500$ , $\delta t = 10^{-4}$ and $MC = 5000$. 
\begin{table}
\begin{center}
\begin{tabular}{|c||l|l|l|c|}
\hline
\multicolumn{5}{|c|}{$c_0 = 1, k = 1$} \\ \hline
$Y$&   $X_{MC}, T = 500$  &   $\frac{\delta_{MC}}{\tau_{MC}}$ & $\tau_{MC}$  &   $\mbox{Relative error}$ $\%$ \\ \hline
0.1  & $0.807^{\pm 0.031}$ & $0.834^{\pm 0.069}$ &  $6.61^{\pm 0.11}$ & $3.2$ \\ \hline
0.2  & $0.649^{\pm 0.026}$ & $0.624^{\pm 0.047}$ & $8.74^{\pm 0.13}$    & $3.8$ \\ \hline
0.3  & $0.493^{\pm 0.020}$ & $0.464^{\pm 0.034}$ & $10.45^{\pm 0.16}$  & $5.8$ \\ \hline
0.4  & $0.361^{\pm 0.014}$ & $0.355^{\pm 0.026}$ & $12.12^{\pm 0.18}$  & $1.7$ \\ \hline
0.5  & $0.266^{\pm 0.011}$ & $0.257^{\pm 0.019}$ & $13.80^{\pm 0.21}$  & $3.3$ \\ \hline
0.6  & $0.195^{\pm 0.008}$ & $0.198^{\pm 0.014}$ & $16.15^{\pm 0.26}$  & $1.5$ \\ \hline
0.7  & $0.137^{\pm 0.005}$ & $0.149^{\pm 0.011}$ & $18.84^{\pm 0.31}$  &  $8$ \\ \hline
0.8  & $0.103^{\pm 0.004}$ & $0.112^{\pm 0.008}$ & $22.80^{\pm 0.39}$  &  $8$ \\ \hline
0.9  & $0.071^{\pm 0.003}$ & $0.086^{\pm 0.006}$ & $26.79^{\pm 0.47}$  & $15$ \\ \hline
\end{tabular}
\caption{Monte-Carlo simulations to compare numerical solution of the left and right hand sides of \eqref{chap3:sigma}, $T = 500$ , $\delta t = 10^{-4}$ and $MC = 5000$.}
\label{chap3:tablesigma}
\end{center}
\end{table}

\begin{remark}
From a numerical point of view, the behavior in long cycles is very relevant. Indeed, computing the left hand side of \eqref{chap3:sigma} is much more expensive in time compared to the right hand side (see the $\tau_{MC}$-column of table \ref{chap3:tablesigma}). 
\end{remark}

\end{document}